\documentclass[preprint]{imsart}

\usepackage{amsthm,amsmath}
\usepackage{amssymb}
\usepackage{xcolor}
\usepackage{tikz}
\RequirePackage[numbers]{natbib}

\usepackage{xcolor}

\startlocaldefs

\newtheorem{lemma}{Lemma}
\newtheorem{theorem}{Theorem}

\newtheorem{definition}{Definition}

\newtheorem{conjecture}{Conjecture}
\newtheorem{corollary}{Corollary}

\newcommand{\nR}{\mathbb{R}}

\newcommand{\nA}{\mathcal{A}}
\newcommand{\nB}{\mathcal{B}}
\newcommand{\nC}{\mathcal{C}}

\newcommand{\nF}{\mathcal{F}}
\newcommand{\nG}{\mathcal{G}}
\newcommand{\nH}{\mathcal{H}}
\newcommand{\nI}{\mathcal{I}}

\newcommand{\bbZ}{\mathbb{Z}}

\newcommand{\ra}{\rightarrow}

\newcommand{\PP}[1]{\mathbb{P}\big[ #1 \big] }

\newcommand{\EE}[1]{\mathbb{E}[ #1 ] }

\newcommand{\EEbb}[1]{\mathbb{E}\Big[ #1 \Big] }

\newcommand{\EEd}[2]{\mathbb{E}_{#2}\left[ #1 \right] }
\newcommand{\EEdnp}[2]{\mathbb{E}_{#2} #1 }

\newcommand{\Ind}[1]{\mathbb{I}\left[ #1 \right] }

\newcommand{\norm}[1]{\left\lVert #1 \right\rVert}
\newcommand{\abs}[1]{\left| #1 \right|}

\newcommand{\mrisk}[1]{R_{mar}(h)}
\newcommand{\hmrisk}[1]{\hat{R}_{mar}(h)}

\newcommand{\binary}{\lbrace 0, 1 \rbrace}

\newcommand{\graph}{\mathcal{G}}

\newcommand{\e}[1]{\exp( #1 )}

\def\ie{{i.e}\onedot}

\usepackage{xspace}
\makeatletter
\DeclareRobustCommand\onedot{\futurelet\@let@token\@onedot}
\def\@onedot{\ifx\@let@token.\else.\null\fi\xspace}

\def\eg{{e.g}\onedot} 
\def\ie{{i.e}\onedot}

\makeatother

\makeatletter
\newcommand{\subalign}[1]{%
  \vcenter{%
    \Let@ \restore@math@cr \default@tag
    \baselineskip\fontdimen10 \scriptfont\tw@
    \advance\baselineskip\fontdimen12 \scriptfont\tw@
    \lineskip\thr@@\fontdimen8 \scriptfont\thr@@
    \lineskiplimit\lineskip
    \ialign{\hfil$\m@th\scriptstyle##$&$\m@th\scriptstyle{}##$\crcr
      #1\crcr
    }%
  }
}
\makeatother

\newcommand{\seqsep}[1]{\vec{\alpha}(#1)}

\endlocaldefs

\begin{document}

\begin{frontmatter}

\title{Dependency-dependent Bounds for Sums of Dependent Random Variables}

\runtitle{Dependency-dependent Bounds}

\begin{aug}
\author{\fnms{Christoph H.} \snm{Lampert}\thanksref{a}\ead[label=e2]{chl@ist.ac.at}},
\author{\fnms{Liva} \snm{Ralaivola}\thanksref{b}\ead[label=e3]{liva.ralaivola@lif.univ-mrs.fr}}
\and
\author{\fnms{Alexander }\snm{Zimin}\thanksref{a}\corref{}\ead[label=e1]{azimin@ist.ac.at}}
\address[a]{IST Austria, Am Campus 1, Klosterneuburg, 3400, Austria. \printead{e2,e1};}
\address[b]{Laboratoire d’Informatique Fondamentale de Marseille, Parc Scientifique et Technologique de Luminy, 163, avenue de Luminy - Case 901, F-13288 Marseille Cedex 9. \printead{e3}}

\runauthor{C.H. Lampert, L. Ralaivola and A. Zimin}

\affiliation{???}

\end{aug}

\begin{abstract}
We consider the problem of bounding large deviations for non-i.i.d. random variables that are allowed to have arbitrary dependencies. 
Previous works typically assumed a specific dependence structure, namely the existence of independent components.
Bounds that depend on the degree of dependence between the observations have only been studied in the theory of mixing processes, where variables are time-ordered.
Here, we introduce a new way of measuring dependences within an unordered set of variables. 
We prove concentration inequalities, that apply to any set of random variables, but benefit from the presence of weak dependencies.
We also discuss applications and extensions of our results to related problems of machine learning and large deviations.
\end{abstract}

\begin{keyword}
\kwd{large deviations bounds}
\kwd{dependent data}
\end{keyword}



\end{frontmatter}

\section{Introduction}

For a set $\nA = \lbrace X_1, \dots, X_n \rbrace$ of a random variables, we study the concentration of their mean, $\frac{1}{n} \sum_{i=1}^{n} X_i$.
When the variables are independent, this is a well-studied topic with numerous results, \eg see \citep{boucheron2013concentration}.
For example, when $0 \leq X_i \leq 1$ for each $i=1,\dots,n$, Hoeffding's inequality~\citep{hoeffding1963probability} 
provides the following bound on the deviations of the sample mean from its expectation. For any $t > 0$:
\begin{equation}
\PP{\frac{1}{n} \sum_{i=1}^{n} X_i - \frac{1}{n} \sum_{i=1}^{n} \EE{X_i} > t} \leq \e{-2n t^2}.
\end{equation}
However, once we alleviate the independence assumption, the situation becomes more complicated. 
A lot of existing research studies 
the case of time series, \ie stochastic processes with integers as index set, meaning that there is a natural ordering inside $\nA$.
For example, in \citep{bosq2012nonparametric} it is shown that if $0 \leq X_i \leq 1$, then 
for any integer $\mu \in [1, \frac{n}{2}]$ and any $t > 0$ with $\nu = \lfloor \frac{n}{2\mu} \rfloor$:
\begin{equation}\label{eq:bosqbound}
\PP{\frac{1}{n}\sum_{i=1}^n X_i - \frac{1}{n} \sum_{i=1}^{n} \EE{X_i} > t} \leq 4 \e{-\frac{\mu t^2}{8}} + 22\mu \alpha_{\nu}\sqrt{1+\frac{4}{t}},
\end{equation}
where $\alpha_{\nu}$ are $\alpha$-mixing coefficients of the process (a definition is given in Section \ref{sec:mixing}).
An important feature of this bound is that it reflects the strength of the dependence between the variables as measured by mixing coefficients.
Unfortunately, this result applies only to stochastic processes, 
while there are a lot of cases when the dependent variables do not have a 
natural ordering, such as the Ising model \citep{ising1925}, where they are 
distributed spatially.

For general sets of random variables, most of the existing concentration results require the existence 
of independent components within $\nA$. 
For example, the following inequality follows from Theorem~2.1 in \citep{Janson01}. For any $t > 0$:
\begin{equation}
\PP{\frac{1}{n} \sum_{i=1}^{n} X_i - \frac{1}{n} \sum_{i=1}^{n} \EE{X_i} > t} \leq \e{-\frac{2n t^2}{\chi(\nA)}},
\end{equation}
where $\chi(\nA)$ is the coloring number of the dependency graph of $\nA$ (see Section~\ref{sec:concentration}).
A shortcoming of this bound, however, is that the dependency graph uses only information about 
the independence of variables, but is oblivious to the strength of any existing dependencies.
%
%
%
%
As it was noticed in \citep{Janson01}, results that ignore this additional information \emph{"can be expected to be wasteful and not give optimal results when the dependencies that exist are weak"}.

Our work combines the best features of these existing approaches.
We prove bounds that apply to general sets of random variables and, at the same time, adjust to the strength of the dependencies between them.

Our first contribution is a suitable definition of a measure of dependence for a set of random variables that controls how closely their 
joint distribution is approximated by the product distribution with the same marginals. 
The tightness of this approximation and its relation to the dependency measure is the content of our central result, Theorem~\ref{theorem:core} (Approximation Theorem) that we prove in Section \ref{sec:corelemma}.
Based on this result, we then prove two new concentration inequalities, stated as Theorem~\ref{theorem:maincovering} and \ref{theorem:mainfractional} that, for example, yield the following bound. 
%
If $0 \leq X_i \leq 1$, then for any $t>0$ and any $\gamma\geq 0$:
\begin{equation}
\PP{\frac{1}{n}\sum_{i=1}^n X_i - \frac{1}{n} \sum_{i=1}^n\EE{X_i} > t } \leq \e{-\frac{nt^2}{8\chi_\gamma(\nA)}} + 18n\gamma\sqrt{\frac{2}{t}},
\end{equation}
where $\gamma$ is an upper bound on the allowed weak dependencies and $\chi_\gamma(\nA)$ is the coloring number 
of the thresholded dependency graph (the graph of all dependencies that exceed $\gamma$). 
Further, we show additional consequences of the Approximation Theorem and apply the obtained results to concrete examples of dependent random variables.

The rest of this paper is organized as follows. 
We review related work in Section~\ref{sec:relatedwork}.
In Section~\ref{sec:corelemma} we introduce the dependency measure and prove the Approximation Theorem.
In Section~\ref{sec:concentration} we go through the background on proper covers for sets of random variables, introduce 
a new notion of soft covers, and prove the concentration bounds.
Some additional consequences of the Approximation Theorem are presented in Section~\ref{sec:consequences}.
We complete the paper with applications of the derived concentration bounds to mixing processes, 
independent cascade models and lattice models in Section~\ref{sec:applications}.

\section{Related work}\label{sec:relatedwork}
As discussed in the introduction, two directions of research are 
most relevant for our result: concentration bounds for stochastic 
processes and concentration bounds based on independent groups
of variables. 

An independent block technique, which goes back to \citep{bernstein1927}, was introduced for stochastic processes in \citep{YuBin01} using $\beta$-mixing coefficients as a measure of dependence.
%
%
Other notions of mixing were also considered, for example, $\eta$-mixing 
in \citep{kontorovich2008concentration} or $\nC$-mixing in \citep{hang2015bernstein}.
Among these works, the most relevant to us are \citep{vidyasagar2005convergence,han2016}, where the authors prove concentration inequalities for exponentially $\alpha$-mixing processes, and, especially, \citep{bosq2012nonparametric}, where an additive bound \eqref{eq:bosqbound} with no restrictions on the rate of mixing is established. 
Mixing coefficients are not the only measure of dependence for stochastic processes, though. For example, in~\citep{Paulin01} 
the author uses a pseudo spectral gap to show the concentration of Markov chains. 

The existence of independent components is a widely used assumption in the literature. 
For example, Stein's method~\citep{stein1986approximate} can be applied to problems 
with dependencies \citep{barbour1989central, ross2011fundamentals}, 
and also to prove concentration \citep{chatterjee2007stein, chatterjee2010applications}.
In \citep{schmidt1995chernoff} the authors used an assumption of $k$-wise independence, where one assumes 
independence only for subsets of size less than $k$, and proved Chernoff-Hoeffding type 
bounds.
This was further relaxed in \citep{impagliazzo2010constructive}, where it was assumed that the expectation of the product of variables of each subset is exponentially small in the size of the corresponding subset.
In \citep{Janson01} an inequality of Hoeffding's type is proved relying on 
the covering properties of the dependence graph. 
As we pointed out in the introduction, the main limitation of these works is that 
they treat dependencies in a binary way, either present or not, being oblivious to 
the strength of dependencies.
A notable exception is \citep{dhurandhar2013auto}, where a bound is shown that depends on some 
measure of dependence within each independent component under certain parametric conditions
on the distribution.
Note that our work is orthogonal to this: we focus on the dependence between components 
and we do not require additional conditions on the distribution.

%

%
A related approach to ours was considered in \citep{feray2016ising} and \citep{feray2016weighted} 
to prove central limit theorems.
The author extended the definition of the dependency graph to weighted dependency graphs 
and used the bounds on the cumulants of the variables as a dependency condition for variables.
%

\section{Approximation Theorem}\label{sec:corelemma}
In this section we state and prove a new approximation theorem for dependent random variables. 
%
%
First, we remind the reader of the notion of $\alpha$-dependence~\citep{Bradley01}.

%
%
\begin{definition}[$\alpha$-dependence] 
Let $(\Omega, \nF, \mathbb{P})$ be a probability space.
Given two sigma algebras $\nB$ and $\nC$ belonging to $\nF$, the \emph{$\alpha$-dependence} coefficient between them is defined as 
\begin{equation}
\alpha(\nB| \nC) = \sup_{B\in\nB, C\in\nC} \abs{\PP{B \cap C} - \PP{B}\PP{C}}.
\end{equation}
Consequently, the $\alpha$-dependence coefficient between two random variables $X$ and $Y$ is defined as 
\begin{equation}
\alpha(X| Y) = \alpha(\sigma(X)|\sigma(Y)),
\end{equation}
where $\sigma(\cdot)$ denotes the $\sigma$-algebra generated by the corresponding random variable.
\end{definition}

Based on this, we introduce a new notion of $\alpha$-separation of a set of random variables.

\begin{definition}[$\alpha$-separation]
The \emph{$\alpha$-separation coefficient} of a set of random variables $\nA = \lbrace X_1, \dots, X_k \rbrace$ is
\begin{equation}\label{eq:seq-alpha-sep}
\seqsep{\nA} = \min_{\pi\in\Pi_k} \frac{1}{k} \sum_{i=1}^{k-1} \alpha(X_{\pi(i)} | \lbrace X_{\pi(i+1)}, \dots, X_{\pi(k)} \rbrace),
\end{equation}
where $\Pi_k$ is the set of all permutations of $1,\dots,k$.
\end{definition}
For examples of $\alpha$-separation, see Section~\ref{sec:applications}.

An important result that lies underneath all results in this paper is the following new theorem that allows us to drop all of the dependencies inside a set of random variables by constructing another set of independent random variables that have the same marginal distributions and approximate the values of the variables of the original set.
The precision of the approximation depends on the $\alpha$-separation coefficient introduced above.
\begin{theorem}[Approximation Theorem]\label{theorem:core}
For a set of random variables $\nA = \lbrace X_1, \dots, X_k \rbrace$ with each $a_i \leq X_i \leq b_i$, let $r(\nA) = \max_{1\leq i \leq n}\abs{b_i-a_i}$ be the maximum range on the variables in the set.
Then there exists another set of random variables $\nA^\star = \lbrace X_1^\star, \dots, X_k^\star \rbrace$ with the following properties:
\begin{enumerate}
\item each $X_i^\star$ has the same marginal distribution as $X_i$,
\item all variables inside $\nA^\star$ are independent,
\item\label{point:approx} For any $\lambda > 0$: $\PP{\frac{1}{k}\sum_{i=1}^{k}\abs{X_i - X_i^\star} > \lambda} \leq 18 k \seqsep{\nA} \sqrt{\frac{r(\nA)}{\lambda}}$.
\end{enumerate}
\end{theorem}
The factor $k$ in Property \ref{point:approx} tells us that, for the theorem to be non-trivial, the dependence has to be smaller than $1/k$.
In most situations that we consider in the paper, the dependence coefficient is even exponentially small in $k$. 
\begin{proof}[Proof of Theorem \ref{theorem:core}]
After potentially reordering the variables inside $\nA$ we can assume that the permutation 
that achieves the minimum inside the definition of $\alpha$-separation coefficient 
in \eqref{eq:seq-alpha-sep} is just the identity mapping. 
%
We now make repeated use of Bradley's result on constructing tight copies 
of single random variables (Theorem \ref{theorem:bradley} in the appendix). 
In words, for a given target variable and a set of another variables this theorem asserts the existence of a copy that has the same marginal distribution as the target variable and is independent of the given set of variables; in addition, the probability that the copy deviates from the target variable is controlled by the $\alpha$-dependence between the target and the given set.
First, we define $X^\star_1$, as a copy of $X_1$ that is independent of $Z_1 = (X_2, \dots, X_k)$.
Second, we construct $X^\star_2$, a copy of $X_2$, which is independent of $Z_2 = (X_1^\star, X_3, \dots, X_k)$, 
then, $X_3^\star$, independent of $Z_3 = (X_1^\star, X_2^\star, X_4, \dots, X_k)$ and so on. 
After $k-1$ steps, we obtain a set $\nA^\star = \lbrace X_1^\star, \dots, X_k^\star \rbrace $ with the $X_k^\star = X_k$.

Now we verify that these variables satisfy the conditions of the lemma.
By construction, each $X_i^\star$ has the same marginal distribution as $X_i$.
It is also easy to see that all constructed variables are independent:
\begin{align}
\PP{X^\star_{1}\!\in\! B_1, \dots, X^\star_{K} \!\in\! B_k} 
& = \PP{ X^\star_{1}\!\in\! B_1, \dots, X^\star_{K-1} \!\in\! B_{k-1} }\PP{X^\star_{k} \!\in\! B_k}  \notag \\
& = \PP{ X^\star_{1}\!\in\! B_1, \dots, X^\star_{k-2} \!\in\! B_{k-2} } \PP{X^\star_{k-1} \!\in\! B_{k-1}} \PP{X^\star_{k} \!\in\! B_k}  \notag \\
&= \dots = \prod_{i=1}^k \PP{X^\star_{i} \!\in\! B_i}
\end{align}
for any rectangle $B_1 \times \dots \times B_k \in \nB(\nR^k)$, where $\nB(\nR^k)$ is the Borel sigma algebra of $\nR^k$.
In addition, at each step, the dependence of $X_i$ on $Z_i$ is bounded by the dependence of $X_i$ on $X_{i+1}, \dots, X_k$.
To see this, observe that any set $B \in \nB(\nR^{k-1})$ can be written as a countable union of disjoint rectangles: $B = \bigcup_{j=1}^\infty B_j$, where each $B_j = \bigotimes_{m=1}^{k-1}B_{j,m}$ with $B_{j,m}\in\nB(\nR)$. 
Then, for any $A \in \nB(\nR)$,
\begin{align}
& \PP{ X_i \in A, Z_i \in B } = \sum_{j=1}^{\infty}\PP{ X_i \in A, Z_i \in B_j } \\
& = \sum_{j=1}^{\infty}\PP{ X_i \in A, \ (X^\star_1, \dots, X^\star_{i-1}) \in \bigotimes_{m=1}^{i-1} B_{j,m}, \ (X_{i+1}, \dots, X_{k}) \in \bigotimes_{m=i}^{k-1} B_{j,m} } \\
& = \sum_{j=1}^{\infty} \PP{ (X^\star_1, \dots, X^\star_{i-1}) \in \bigotimes_{m=1}^{i-1} B_{j,m} } \PP{ X_i \in A, \ (X_{i+1}, \dots, X_{k}) \in \bigotimes_{m=i}^{k-1} B_{j,m} }.
\end{align}

Using the same decomposition for $\PP{Z_i \in B }$,
\begin{align}
& \PP{ X_i \in A, Z_i \in B } - \PP{ X_i \in A} \PP{ Z_i \in B } = \sum_{j=1}^{\infty} \PP{ (X^\star_1, \dots, X^\star_{i-1}) \in \bigotimes_{m=1}^{i-1} B_{j,m} } \times \\
& \big(  \PP{ X_i \in A, (X_{i+1}, \dots, X_{k}) \in \bigotimes_{m=i}^{k-1} B_{j,m} } - \PP{ X_i \in A } \PP{ (X_{i+1}, \dots, X_{k}) \in \bigotimes_{m=i}^{k-1} B_{j,m} } \big) \notag \\
& \leq \sum_{j=1}^{\infty} \PP{ (X^\star_1, \dots, X^\star_{i-1}) \in \bigotimes_{m=1}^{i-1} B_{j,m} } \alpha(X_i| X_{i+1}, \dots, X_k) \leq \alpha(X_i| X_{i+1}, \dots, X_k).
\end{align}
Hence, from Theorem \ref{theorem:bradley}, we have for each $X_i^\star$,
\begin{align}
\PP{\abs{X_i - X^\star_i} > \lambda} \leq 18 \sqrt{\frac{b_i-a_i}{\lambda}} \alpha(X_i|Z_i) \leq 18 \sqrt{\frac{b_i-a_i}{\lambda}} \alpha(X_i| X_{i+1}, \dots, X_k). \label{eq:copybound}
\end{align}
And, finally, we can bound
\begin{align}
\PP{\frac{1}{k}\sum_{i=1}^{k}\abs{X_i - X_i^\star} > \lambda} & \leq \sum_{i=1}^{k} \PP{\abs{X_i - X_i^\star} > \lambda} \\
& \leq 18 \frac{1}{\sqrt{\lambda}} \sum_{i=1}^{k-1}\alpha(X_i| X_{i+1}, \dots, X_k)\sqrt{b_i-a_i} \\
& = 18k\seqsep{\nA}\sqrt{\frac{r(\nA)}{\lambda}}.
\end{align}

\end{proof}


\section{Concentration results}\label{sec:concentration}

In this section we present our main application of the Approximation Theorem to the concentration of averages of random variables.
Our results rely on the new notion of $\gamma$-independence and soft covers 
that 
allow us to identify and manipulate subsets of random variables that are weakly dependent.
%

\subsection{Proper covers}
We start with the basic notion of covers and fractional covers for a set $\nA$ of random 
variables, which were formally defined, \eg, in~\citep{Janson01}.
\begin{definition}[Proper cover of $\nA$]
Let $\nA$ be a set $\lbrace X_1, \dots, X_k \rbrace$ of random variables.
\begin{itemize}
\item a set $\nI \subseteq [k]$ is called independent if the corresponding random variables $\lbrace X_i, i\in\nI \rbrace$ are independent,
\item a family $\lbrace \nI_j \rbrace$ of subsets of $[k]$ is a cover of $\nA$ if $\cup_{j}\nI_j = [k]$,
\item a cover is {\em proper} if each set $\nI_j$ in it is independent,
\item $\chi(\nA)$ is the size of the smallest proper cover of $\nA$, \ie the smallest $m$ such that $[k]$ is the union of $m$ independent subsets.
\end{itemize}
\end{definition}
\begin{definition}[Proper fractional cover of $\nA$]
Let $\nA$ be a set $\lbrace X_1, \dots, X_k \rbrace$ of random variables.
\begin{itemize}
\item a family $\{(\nI_j, w_j )\}$ of pairs $(\nI_j, w_j)$, where $\nI_j\subseteq [k]$ and $w_j\in[0, 1]$, is a {\em fractional cover} of $\nA$ if for all $1 \leq i \leq k$, $\sum_{j:i\in\nI_j} w_j \geq 1$,
\item a fractional cover is {\em proper} if each set $\nI_j$ in it is independent,
\item $\chi^*(\nA)$ is the minimum of $\sum_j w_j$ over all proper fractional covers $\{(\nI_j, w_j)\}_j$ of $\nA$.
\end{itemize}
\end{definition}
Note that, as observed by \citep[Lemma 3.2]{Janson01}, we can restrict ourselves to working with {\em exact} fractional covers, 
which for every $i\in[k]$ requires $\sum_{j:i\in\nI_j} w_j = 1$ instead of the weaker condition $\sum_{j:i\in\nI_j} w_j \geq 1$ 
for non-exact fractional covers. This is possible without loss of generality, as any fractional covers induces an exact fractional 
cover.

A cover of $\nA$ splits the set into subsets of independent variables so that the union of all the subsets is the original set.
The usefulness of this decomposition is that it makes it possible to have concentration results for sets of independent variables to be extended to sets of possibly dependent variables by i) using the results on the smaller independent subsets and by ii) combining the so obtained results in a global result applying for the whole family of variables.

An alternative way to get a good grasp at what a proper cover is, is to connect it with 
the idea of graph coloring.
We can define a dependency graph $G$ for a set $\nA = \lbrace X_1, \dots, X_k \rbrace$ as a graph with vertex set $[k]$ and if $i$ is not connected by an edge to any element of a set $\nI \subseteq [k]$, then $X_i$ is independent of the variables $X_j$ for $j \in \nI$.
The coloring number of this dependency graph is an upper bound on $\chi(\nA)$. 
%
%
%
%
Similarly, there exist a notion of fractional coloring of a graph and the corresponding fractional coloring number that provides an upper bound on $\chi^\star(\nA)$.
%
%

\subsection{Soft covers: beyond 0-1 dependencies}\label{sec:softcovers}
In this paper we look beyond the proper covers and allow for sets that are interdependent, 
but with a controlled amount of dependence, as measured by $\alpha$-separation.
By taking this route, on the contrary to what was studied in~\citep{Janson01}, we will be able to 
establish results that take into account the magnitude of the dependencies between random 
variables. 

%

%
%
%
%
\begin{definition}[Soft cover of $\nA$]
Let $\nA$ be a set of random variables and $\gamma \geq 0$ be a threshold.
\begin{itemize}
\item a set $\nI \subseteq [k]$ is called \emph{$\gamma$-independent} if $\seqsep{\nI} \leq \gamma$,
\item a cover is called \emph{soft} if each set $\nI_j$ in it is $\gamma$-independent,
\item $\chi_\gamma(\nA)$ is the size of the smallest soft cover of $\nA$ for a given $\gamma$.
\end{itemize}
\end{definition}
Analogously, we extend the definition of fractional covers.
\begin{definition}
Let $\nA$ be a set of random variables and $\gamma \geq 0$ a threshold.
\begin{itemize}
\item a fractional cover is called \emph{soft} if each set $\nI_j$ in it is $\gamma$-independent,
\item $\chi^\star_\gamma(\nA)$ is the minimum of $\sum_{j} w_j$ over all soft fractional covers $\{(\nI_j, w_j)\}_j$ of $\nA$ for a given $\gamma$.
\end{itemize}
\end{definition}
\usetikzlibrary{patterns}
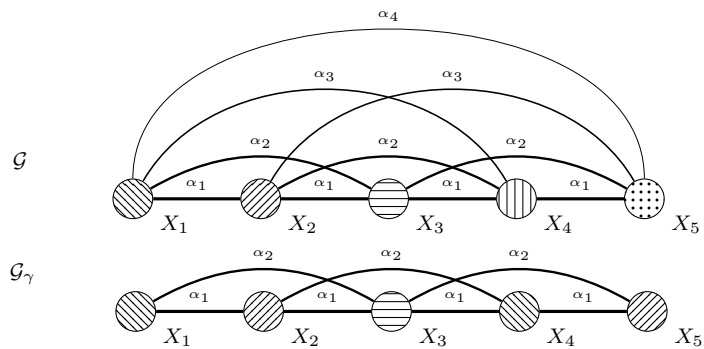
\begin{figure}
\begin{center}
	\begin{tikzpicture}[style1/.style={circle,draw,minimum size=15,pattern=north west lines},style2/.style={circle,draw,minimum size=15,pattern=north east lines},style3/.style={circle,draw,minimum size=15,pattern=horizontal lines},style4/.style={circle,draw,minimum size=15,pattern=vertical lines},style5/.style={circle,draw,minimum size=15,pattern=dots
}]
	\pgfmathsetmacro{\nn}{4}
	
	\node at (-1.5,0.5) {$\graph$};
	
	\foreach \i in {0,1,...,\nn}{
		\pgfmathtruncatemacro{\ii}{\i+1}
		\pgfmathtruncatemacro{\st}{Mod(\i, 5)+1}
		\node[label={[label distance=0.01cm]-25:${X}_{\ii}$},style\st] (\i) at (\i*1.7,0) {};
	}
	
	\pgfmathsetmacro{\step}{30}
	\pgfmathsetmacro{\maxI}{\nn-1}
	\foreach \i in {0,1,...,\maxI}{
		\pgfmathtruncatemacro{\z}{\i+1}
	    \foreach \j in {\z,...,\nn}{
					\pgfmathtruncatemacro{\delta}{\j-\i}
	    			\pgfmathsetmacro{\angle}{\step*(\delta-1)}
	    			\pgfmathsetmacro{\thickness}{0.3*(\nn-\delta)+0.3}
					\draw[line width=\thickness pt,out=\angle,in=180-\angle] (\i) to node[midway,above] {\tiny $\alpha_\delta$} (\j);
		}
	}
	\end{tikzpicture}
	
	\begin{tikzpicture}[style1/.style={circle,draw,minimum size=15,pattern=north west lines},style2/.style={circle,draw,minimum size=15,pattern=north east lines},style3/.style={circle,draw,minimum size=15,pattern=horizontal lines},style4/.style={circle,draw,minimum size=15,pattern=north west lines},style5/.style={circle,draw,minimum size=15,pattern=north east lines}]
	\pgfmathsetmacro{\nn}{4}
	\pgfmathsetmacro{\nb}{1}

	\node at (-1.5,0.5) {$\graph_{\gamma}$};
		
	\foreach \i in {0,1,...,\nn}{
		\pgfmathtruncatemacro{\ii}{\i+1}
		\pgfmathtruncatemacro{\st}{Mod(\i, 5)+1}
		\node[label={[label distance=0.01cm]-25:${X}_{\ii}$},style\st] (\i) at (\i*1.7,0) {};
	}
	
	\pgfmathsetmacro{\step}{30}
	\pgfmathsetmacro{\maxI}{\nn-1}
	\foreach \i in {0,1,...,\maxI}{
		\pgfmathtruncatemacro{\z}{\i+1}
	    \foreach \j in {\z,...,\nn}{
					\pgfmathtruncatemacro{\delta}{\j-\i}
	    			\pgfmathsetmacro{\angle}{\step*(\delta-1)}
	    			\pgfmathsetmacro{\thickness}{0.3*(\nn-\delta)+0.3}
					\draw[line width=\thickness pt,out=\angle,in=180-\angle] (\i) to node[midway,above] {\tiny $\alpha_\delta$} (\j);
					\ifnum \delta > \nb
						\breakforeach
					\fi 
		}
	}
	\end{tikzpicture}
\end{center}
\caption{Typical graph associated with a mixing process. The nodes correspond to random variables and the edges to dependences.
The strength $\alpha_i$ of dependence depends upon the 'distance' $i$ between the nodes: here, the weaker the dependence, the
thinner the edge. The graph $G_\gamma$ is shown for $\alpha_3 \leq \gamma \leq \alpha_2$. The nodes are patterned according to the coloring.}
\label{figure:graphexample}
\end{figure}


Similarly to the usual covers we can connect these definition to the graph colorings of a carefully defined dependency graph.
We define a thresholded dependency graph $G_\gamma$ for a set $\nA = \lbrace X_1\, \dots, X_k \rbrace$ 
as a graph with vertex set $[k]$ and the edge set defined by the rule: if a node $i$ is not connected 
by an edge to any element of a set $\nI \subseteq [k]$, then $\alpha(X_i| X_j, j\in\nI) \leq \gamma$.
Then any set that is independent in the graph theoretical sense corresponds to a $\gamma$-independent set of random variables.
Consequently, the $\chi_\gamma$ for such a set can be bounded by the coloring number of $G_\gamma$ and a similar observation holds for fractional covers.
As an example of an advantage that this relaxation brings, let us consider a mixing process of length $5$ as depicted in the Figure \ref{figure:graphexample}.
The numbers $\alpha_i$, which represent the strength of dependencies, are formally defined in the Section \ref{sec:mixing}.
As the dependency graph for this process is a complete graph, its coloring number is the maximum possible, that is $5$.
However, we would like it to be as small as possible as the coloring number $1$ corresponds to an independent set.
Any $\gamma$ between $\alpha_2$ and $\alpha_3$ gives us the thresholded graph $G_\gamma$, which now has a chromatic number of $3$, as there are three independent sets: $\lbrace 1, 4\rbrace$, $\lbrace 2, 5\rbrace$ and $\lbrace 3 \rbrace $.
%

\subsection{The concentration}

Having established the dependence measure and the notion of coverings, 
we present a number of concentration bounds that take the strength of 
the dependence into account.
We state the results for covers and fractional covers separately, 
since the proof for the latter uses an additional approximation 
step that leads to an additional term in the bound
%

\begin{theorem}\label{theorem:maincovering}
Let $\nA = \lbrace X_1, \dots, X_n\rbrace$ be a set of random variable with $a_i \leq X_i \leq b_i$ for some $a_i, b_i \in \nR$.
Then, for any threshold $\gamma > 0$, $t > 0$ and $\lambda < t$
\begin{equation}
\PP{\frac{1}{n}\sum_{i=1}^n X_i - \frac{1}{n} \sum_{i=1}^n\EE{X_i} > t } \leq \e{-\frac{2n^2(t-\lambda)^2}{\chi_\gamma(\nA) \sum_{i=1}^n(b_i-a_i)^2}} + 18n\gamma\sqrt{\frac{r(\nA)}{\lambda}}.
\end{equation}
\end{theorem}
The parameter $\lambda$ in Theorem~\ref{theorem:maincovering} provides a trade-off between two terms of the upper bound.
In applications we always use $\lambda = t/2$.
\begin{theorem}\label{theorem:mainfractional}
Fix a set of random variables $\nA = \lbrace X_1, \dots, X_n\rbrace$, with $a_i \leq X_i \leq b_i$ for some $a_i, b_i \in \nR$.
Then for any threshold $\gamma > 0$, $t > 0$ and $\lambda < t$
\begin{align}
\PP{\frac{1}{n}\sum_{i=1}^n X_i - \frac{1}{n} \sum_{i=1}^n\EE{X_i} > t } & \leq \e{-\frac{n^2(t-\lambda)^2}{2\chi_\gamma^\star(\nA) \sum_{i=1}^n(b_i-a_i)^2}} \\
& + \e{-\frac{n^2(t-\lambda)^2}{ 2\sum_{i=1}^n(b_i-a_i)^2}} + 18n\gamma\sqrt{\frac{r(\nA)}{\lambda}}. \notag
\label{eq:mainfractional}
\end{align}
This can be further upper bounded by a less tighter expression
\begin{equation}
\PP{\frac{1}{n}\sum_{i=1}^n X_i - \frac{1}{n} \sum_{i=1}^n\EE{X_i} > t } \leq 2\e{-\frac{n^2(t-\lambda)^2}{2\chi_\gamma^\star(\nA) \sum_{i=1}^n(b_i-a_i)^2}} + 18n\gamma\sqrt{\frac{r(\nA)}{\lambda}}.
\end{equation} 
\end{theorem}

Note that the same results hold for the $\PP{\sum_{i\in\nA}X_i - \sum_{i\in\nA}\EE{X_i} < - t }$.
Obviously, the case $\gamma = 0$ brings us back to the setting of \citep{Janson01}, recovering its bounds
up to constants. 

%
%


\begin{proof}[Proof of Theorem \ref{theorem:maincovering}]
Without loss of generality, we assume that $\EE{X_i} = 0$ for all $i$.
We start by fixing a covering $\left\lbrace \nI_j \right\rbrace$ of $\Gamma(\gamma) $ and constructing another set of random variables, $\lbrace X_1^\star, \dots, X_n^\star \rbrace$,
by applying Theorem \ref{theorem:core} to each set $\nI_j$ in the covering. 
By construction, the variables corresponding to each set $\nI_j$ are independent.

%
We introduce events $E_j = \big\lbrace \frac{1}{|\nI_j|} \sum_{i\in\nI_j} \abs{ X_i - X^\star_i } < \lambda \big\rbrace $ and note that $\PP{E^c_j}$ is bounded by $\frac{18|\nI_j|\seqsep{\nI_j}\sqrt{r(\nA)}}{\sqrt{\lambda}}$ and therefore by $\frac{18|\nI_j|\gamma\sqrt{r(\nA)}}{\sqrt{\lambda}}$.
For $E = \bigcap_{j}E_j$ we get
\begin{align}
\PP{\frac{1}{n}\sum_{i=1}^n X_i > t} & = \PP{\frac{1}{n}\sum_{i=1}^nX_i > t \wedge E} + \PP{\frac{1}{n}\sum_{i=1}^n X_i > t \wedge E^c} \\
& \leq \PP{\frac{1}{n}\sum_{i=1}^nX_i - \frac{1}{n}\sum_{i=1}^nX^\star_i > t - \frac{1}{n}\sum_{i=1}^nX^\star_i \wedge E} + \PP{E^c} \\
& \leq \PP{\lambda > t - \frac{1}{n}\sum_{i=1}^nX^\star_i \wedge E} + \sum_{j}\PP{E^c_j} \\
& \leq \PP{\frac{1}{n}\sum_{i=1}^nX^\star_i > t - \lambda} + \frac{18n\gamma\sqrt{r(\nA)}}{\sqrt{\lambda}}.
\end{align}
We get the statement of the theorem by applying Theorem 2.1 of \citep{Janson01} to the first summand.
\end{proof}
\begin{proof}[Proof of Theorem \ref{theorem:mainfractional}]
The proof of theorem relies on the same construction as in Theorem \ref{theorem:maincovering}, but with an additional approximation step to account for the fact that we will construct several copies of the same variable.
As before, we assume that $\EE{X_i} = 0$ for all $i$.
We start by fixing the fractional cover of $\nA$, $\lbrace \nI_j, w_j \rbrace$.
For each variable $X_i$ we construct a number of different copies, $X_{i,j}^\star$, using Theorem \ref{theorem:core} in a such way that for a fixed $j$, $\lbrace X_{i,j}^\star, i \in \nI_j \rbrace$ are independent.
Moreover, from the proof of Theorem \ref{theorem:core} from \eqref{eq:copybound}, for each $i,j$ we know that $\PP{\abs{X_i - X_{i,j}^\star} > \lambda} \leq 18\sqrt{\frac{b_i-a_i}{\lambda}} \alpha_{i,j}$, where $\alpha_{i,j}$ is the corresponding dependence coefficient for the optimal ordering inside $\nI_j$.
Since we consider exact fractional covers, we know that for any $i$, $\sum_{j:i\in\nI_j}w_j = 1$ and we can use this to define random variables $J_i$ that take the index of the sets from the cover that $X_i$ belongs to with probability $w_j$.
Note that we can require those variables to be independent from each other and from all other random variables we consider.
Finally, we define our final approximations $\tilde{X}_i = X^\star_{i,J_i}$.
The first thing to observe is that
\begin{equation}
\PP{\big| X_i - \tilde{X_i} \big| > \lambda} = \sum_{j:i\in\nI_j} w_j \PP{ \abs{X_i - X_{i,j}^\star} > \lambda}\leq \frac{18}{\sqrt{\lambda}} \sum_{j:i\in\nI_j} w_j  \alpha_{i,j}\sqrt{b_i-a_i}.
\end{equation}
As before, we define events $E_i = \lbrace | X_i - \tilde{X_i} | \leq \lambda \rbrace$ (but now individually for each variable) and $E = \bigcap_{i=1}^n E_i$.
Now we have,
\begin{equation}\label{eq:fract-proof-after-approx}
\PP{\frac{1}{n}\sum_{i=1}^{n} X_i > t} \leq \PP{\frac{1}{n}\sum_{i=1}^{n} \tilde{X_i} > t-\lambda} + \sum_{i=1}^{n}\PP{E_i^c}.
\end{equation}
For the first summand we have the following inequality 
%

\begin{align}
& \PP{\frac{1}{n}\sum_{i=1}^{n} \tilde{X_i} > t-\lambda} \\
& \leq \PP{\frac{1}{n}\sum_{i=1}^{n} \tilde{X_i} - \frac{1}{n}\sum_{i=1}^{n}\sum_{j:i\in\nI_j} w_j X^\star_{i,j} > \frac{t-\lambda}{2}} + \PP{\frac{1}{n}\sum_{i=1}^{n}\sum_{j:i\in\nI_j} w_j X^\star_{i,j} > \frac{t-\lambda}{2}}. \label{eq:fract-proof-endline}
\end{align}
For fixed $X^\star_{i,j}$'s the first term in \eqref{eq:fract-proof-endline} can be bounded using the Hoeffding's inequality for $J_i$ as they are independent.
The second term is bounded by Theorem~2.1 of \citep{Janson01} (skipping the first step of the proof of the theorem).
The last thing that is left to do is to bound the second summand in \eqref{eq:fract-proof-after-approx}.
For it we have
\begin{align}
\sum_{i=1}^{n}\PP{E_i^c} & \leq \frac{18}{\sqrt{\lambda}} \sum_{i=1}^{n} \sum_{j:i\in\nI_j} w_j \alpha_{i,j}\sqrt{b_i-a_i} \\
& \leq \frac{18\sqrt{r(\nA)}}{\sqrt{\lambda}} \sum_j w_j |\nI_j| \seqsep{\nI_j} \\
& \leq \frac{18\sqrt{r(\nA)}\gamma}{\sqrt{\lambda}} \sum_j w_j |\nI_j| = \frac{18\sqrt{r(\nA)}n\gamma}{\sqrt{\lambda}}.
\end{align}

\end{proof}

\subsection{Lower bound}
The goal of this section is to show that the additive linear dependence in Theorems~\ref{theorem:maincovering} and~\ref{theorem:mainfractional}
is unavoidable. 
For this, we demonstrate a lower bound for the concentration in terms of the $\alpha$-separation coefficient. 
%
%
%
%
\begin{theorem}\label{theorem:lowerbound}
For even $n$ and for any integer $t \in [0, \frac{n}{8}]$ and $\gamma \in [0, \frac{1}{4n}]$, there exists a distribution over a set $\nA = \lbrace X_1, \dots, X_n \rbrace$ with each $X_i$ being a Bernoulli random variable with parameter $\frac{1}{2}$ with $\seqsep{\nA} = \gamma$, such that 
\begin{equation}
\PP{\sum_{i=1}^{n}X_i - \frac{n}{2} \geq t} \geq \frac{1}{15}e^{-16t^2/n} + 4n\seqsep{\nA}\frac{1}{2^n}\binom{n-1}{\frac{n}{2}+t-1}.
\end{equation}
\end{theorem}
The proof can be found in the appendix.

For comparison, we apply Theorem \ref{theorem:maincovering} to the variables in the Theorem \ref{theorem:lowerbound} with $\lambda = \frac{t}{2}$ 
and for the case, when we drop all the dependencies at the same time:
\begin{equation}
\PP{\sum_{i=1}^n X_i - \frac{n}{2} > t } \leq e^{-\frac{t^2}{2n}} + 18\sqrt{2}n\seqsep{\nA}\frac{1}{\sqrt{t}}.
\end{equation}
We observe that both right hand sides have the same structure, thereby confirming that the dependence on the coefficient in Theorem \ref{theorem:maincovering} is of the right order.
Moreover, even the seemingly complicated term $g_n(t) = \frac{1}{2^n}\binom{n-1}{\frac{n}{2}+t-1}$ is a lower bound on the $\frac{1}{\sqrt{t}}$, since $\binom{n-1}{\frac{n}{2}+t-1} \sim \frac{2 ^n e^{-2t^2/n}}{\sqrt{\pi n / 2}}$ and, hence, $g_n(t) \sim \frac{e^{-2t^2/n}}{\sqrt{\pi n / 2}} \leq \frac{1}{\sqrt{n}} < \frac{1}{\sqrt{t}}$.

\section{Other consequences of the Approximation Theorem}\label{sec:consequences}
In this section we present two other consequences of the Approximation Theorem that can be of independent interest.
\subsection{$L_p$-Distance for approximations}
The goal of this section is to provide bounds on the $L_p$-distance between the original set of variables and the approximating set.
We formulate this as the following corollary of Theorem \ref{theorem:core}.
\begin{corollary}
For a multivariate random variable $X = (X_1, \dots, X_k)$ with each $a_i \leq X_i \leq b_i$, the approximation $X^\star = (X_1^\star, \dots X_k^\star)$, 
where $X_i^\star$ are constructed according to Theorem \ref{theorem:core}, satisfies for any $p \in [1, \infty)$
\begin{equation}
\norm{X-X^\star}_p \leq r(\nA) \left( \frac{18 p } {p-1/2} \seqsep{\nA} \right)^\frac{1}{p}.
\end{equation}
\end{corollary}
\begin{proof}
Without loss of generality, we assume that the optimum at the definition of $\alpha$-separation coefficient is achieved by the identity mapping.
Then we have
\begin{align}
\EE{\abs{X_i-X_i^\star}^p} & = p \int_{0}^{\infty}t^{p-1}\PP{\abs{X_i-X_i^\star} > t} dt \\
& = p \int_{0}^{r(\nA)}t^{p-1}\PP{\abs{X_i-X_i^\star} > t} dt \\
& \leq p \int_{0}^{r(\nA)}t^{p-1} \frac{18 \alpha(X_i|X_{i+1}, \dots, X_n)\sqrt{r(\nA)}}{\sqrt{t}} dt \\
& = 18 p \alpha(X_i|X_{i+1}, \dots, X_n)r(\nA)^p \int_{0}^{1}t^{p-3/2} dt \\
& = \frac{18 p \alpha(X_i|X_{i+1}, \dots, X_n)r(\nA)^p} {p-1/2}.
\end{align}
Hence, we can conclude that
\begin{equation}
\norm{X-X^\star}_p = \left(\EE{\sum_{i=1}^{k} \abs{X_i - X_i^\star}^p }\right)^\frac{1}{p} \leq \left( \frac{18 p \seqsep{\nA}r(\nA)^p} {p-1/2} \right)^\frac{1}{p}.
\end{equation}
\end{proof}

\subsection{Concentration of variance}
In this section we study the concentration of functions of dependent random 
variables and, as an example, we prove a concentration bound on the variance.
For a set $\nA = \lbrace X_1, \dots, X_n \rbrace$ and a function $f:\nR\ra\nR$, 
we define the set $f(\nA) = \lbrace f(X_1), \dots, f(X_n) \rbrace$.
An important fact is that $\seqsep{f(\nA)} \leq \seqsep{\nA}$, which follows from the definition of the $\alpha$-dependence.
Therefore, the concentration bounds for $f(\nA)$ can be stated in terms of the dependence characteristics of the original set.
\begin{corollary}\label{cor:fA}
For any set $\nA = \lbrace X_1, \dots, X_n \rbrace$ and function $f: \nR \ra \nR$,
such that $a_i \leq f(X_i) \leq b_i$ for some $a_i,b_i\in\nR$. 
Then, for any $t > 0$,
\begin{equation}
\PP{\frac{1}{n}\sum_{i=1}^{n}f(X_i) - \frac{1}{n}\sum_{i=1}^{n}\EE{f(X_i)}} \leq \e{-\frac{2n^2(t-\lambda)^2}{\chi_\gamma(\nA) \sum_{i=1}^n(b_i-a_i)^2}} + 18n\gamma\sqrt{\frac{r(f(\nA))}{\lambda}}.
\end{equation}
\end{corollary}
%
%
As a consequence, we can prove a concentration bound on the estimator of the variance.
\begin{theorem}\label{theorem:variancebound}
Let $\nA = \lbrace X_1, \dots, X_n \rbrace$ be a set of random variables, 
where all $X_i$ take values in $[0,1]$ and have the same mean $\mu$ and variance $\sigma^2$. 
We define $\hat{\mu} = \frac{1}{n}\sum_{i=1}^{n}X_i$ and $\hat{\sigma}^2 = \frac{1}{n}\sum_{i=1}^{n}(X_i - \hat{\mu})^2$.
%
Then, for any $t>0$,
\begin{equation}
\PP{\hat{\sigma}^2 - \sigma^2 > t} \leq 2\e{-\frac{n(t-\lambda)^2}{8\chi_\gamma(\nA)}} + \frac{36n\gamma}{\sqrt{\lambda}}.
\end{equation}
\end{theorem}
\begin{proof}
First, we notice that $\hat{\sigma}^2 = \frac{1}{n}\sum_{i=1}^{n} X_i^2 - \hat{\mu}^2$ and $\sigma = \EE{X_1^2} - \mu^2$.
Therefore,
\begin{align}
\PP{\hat{\sigma}^2 - \sigma^2 > t} & \leq \PP{\frac{1}{n}\sum_{i=1}^{n} X_i^2 - \EE{X_1^2} > \frac{t}{2}} + \PP{\hat{\mu}^2 - \mu^2 > \frac{t}{2}} \\
& \leq \PP{\frac{1}{n}\sum_{i=1}^{n} X_i^2 - \EE{X_1^2} > \frac{t}{2}} + \PP{\hat{\mu} - \mu > \frac{t}{4}}.
\end{align}
We apply Corollary \ref{cor:fA} for the first summand and Theorem \ref{theorem:maincovering} for the second.
\end{proof}

\subsection{Supremum of Lipschitz functions}
This example illustrates an extension to the concentration of different functions of the sample.
In particular, we consider the supremum over a class of Lipschitz functions.
The expressions of this type appear, \eg, in machine learning theory and can be used to prove Rademacher complexity bounds~\citep{shalev2014understanding}.

We give an example using the regular covers to simplify the presentation.
The analogous result can be shown for the fractional covers as well.
For a set $\nA = \lbrace X_1, \dots, X_k \rbrace$, an optimal soft cover $\lbrace \nI_j \rbrace$ with threshold $\gamma$ and a function class $\nF$ define
\begin{equation}
\mathcal{R}_\gamma = \frac{1}{n} \sum_j \EEd{\sup_{f\in\nF} \sum_{i\in\nI_j} (f(X_i) - \EE{X_i})}{\otimes},
\end{equation}
where $\EEdnp{}{\otimes}$ means that the expectation is taken with respect to the distribution constructed as a product of marginals, \ie treating each variable as independent.
The next theorem shows that the supremum of Lipschitz functions concentrates around the $\mathcal{R}_\gamma$.
\begin{theorem}
Let $\nF \subset \left\lbrace f: \nR \ra [0,B] : \abs{f(x) - f(y)} \leq L \abs{x-y} \right\rbrace$ be a set of $L$-Lipschitz functions.
For a fixed set of random variables $\nA = \left\lbrace X_1, \dots, X_n \right\rbrace $ introduce
\begin{equation}
\Phi(\nA) = \sup_{f\in\nF}\frac{1}{n}\sum_{i=1}^n(f(X_i) - \EE{f(X_i)}).
\end{equation}
Then for any threshold $\gamma >0$ and $t\in[0, B]$
\begin{equation}
\PP{ \Phi(\nA) - \mathcal{R}_\gamma  > t } \leq \e{ -\frac{ n t^2}{2\chi_\gamma(\nA)B^2} } + 18 \gamma n\sqrt{\frac{2Lr(\nA)}{t}}.
\end{equation}

\end{theorem}

%
\begin{proof}
For the proof we take the set $\nA^\star = \lbrace X_1^\star, \dots, X_1^\star \rbrace$ constructed as in Theorem \ref{theorem:maincovering}.
The proof follows by observing that $\Phi(\nA) \leq \Phi(\nA^\star)  + \sup_{f\in\nF} \frac{1}{n}\sum_{i=1}^n(f(X_i) - f(X^\star_i))$ and, hence, for any $0 < \lambda < t$
\begin{equation}
\PP{ \Phi(\nA) - \mathcal{R}_\gamma > t } \leq \PP{ \Phi(\nA^\star) - \mathcal{R}_\gamma > t - L\lambda } + 18 \gamma n \sqrt{\frac{r(\nA)}{\lambda}}.
\end{equation}
We can further upper bound $\Phi(\nA^\star) \leq \frac{1}{n} \sum_j \sup_{f\in\nF} \sum_{i\in\nI_j} (f(X^\star_i) - \EE{X_i})$ and then the proof follows from Theorem 2 of \citep{usunier2005generalization} and by setting $\lambda = \frac{t}{2L}$.
\end{proof}
\section{Applications}\label{sec:applications}
%
In this section we present three applications of our analysis to different examples of dependent random variables. 
For simplicity, we use only regular covers and Theorem \ref{theorem:maincovering}. Analogous results 
for fractional covers can be achieved by use of Theorem \ref{theorem:mainfractional}.

\subsection{Mixing processes}\label{sec:mixing}
\begin{figure}
\newcommand{\boxnum}{13}
\begin{center}
	\begin{tikzpicture}
	[
	style1/.style={circle,draw,minimum size=15,pattern=north west lines},
	style2/.style={circle,draw,minimum size=15},
	style3/.style={circle,draw,minimum size=15,pattern=horizontal lines},
	style4/.style={circle,draw,minimum size=15},
	style5/.style={circle,draw,minimum size=15,pattern=north east lines},
	style6/.style={circle,draw,minimum size=15},
	style7/.style={circle,draw,minimum size=15,pattern=dots},
	]
	\node at (-1.5,0.2) {(a)};
	
	\pgfmathsetmacro{\length}{10}
	\pgfmathsetmacro{\dashes}{0.1}
	\pgfmathsetmacro{\boxheight}{0.5}
	\pgfmathtruncatemacro{\n}{\boxnum}
	\pgfmathsetmacro{\boxlength}{\length/\n}
	
	\foreach \i in {0,1,...,\n}{
		\pgfmathtruncatemacro{\ii}{\i+1}
		\pgfmathtruncatemacro{\st}{Mod(\i/2, 7)+1}
		\node[style\st] (\i) at (\i*0.7,0) {};
	}

	\pgfmathtruncatemacro{\nn}{\n-1}
	
	\foreach \i in {0,1,...,\nn}{
		\pgfmathtruncatemacro{\ii}{\i+1}
		\draw[line width=1pt,out=0,in=180] (\i) to node[midway,above] {} (\ii);
	}
	
	\end{tikzpicture}
	
	\begin{tikzpicture}
	[
		style1/.style={circle,draw,minimum size=15,pattern=north west lines},
		style2/.style={circle,draw,minimum size=15,pattern=horizontal lines},
		style3/.style={circle,draw,minimum size=15,pattern=dots}
	]
	\node at (-1.5,0.2) {(b)};
	
	\pgfmathsetmacro{\length}{10}
	\pgfmathsetmacro{\dashes}{0.1}
	\pgfmathsetmacro{\boxheight}{0.5}
	\pgfmathtruncatemacro{\n}{\boxnum}
	\pgfmathsetmacro{\boxlength}{\length/\n}
	
	\foreach \i in {0,1,...,\n}{
		\pgfmathtruncatemacro{\ii}{\i+1}
		\pgfmathtruncatemacro{\st}{Mod(\i, 3)+1}
		\node[style\st] (\i) at (\i*0.7,0) {};
	}

	\pgfmathtruncatemacro{\nn}{\n-1}
	
	\foreach \i in {0,1,...,\nn}{
		\pgfmathtruncatemacro{\ii}{\i+1}
		\draw[line width=1pt,out=0,in=180] (\i) to node[midway,above] {} (\ii);
	}
	
	\end{tikzpicture}
	
	\begin{tikzpicture}
	
	\node at (-1.9,0.4) {};
	
	\pgfmathsetmacro{\length}{10}
	\pgfmathsetmacro{\dashes}{0.1}
	\pgfmathsetmacro{\boxheight}{0.5}
	\pgfmathtruncatemacro{\n}{\boxnum}
	\pgfmathsetmacro{\boxlength}{\length/\n}
	
	\foreach \i in {0,1,...,\n}{
		\pgfmathtruncatemacro{\ii}{\i+1}
		\node at (\i*0.7, 0) {${X}_{\ii}$};
	}
	
	\end{tikzpicture}
\end{center}
\caption{A difference between (a) the splits of a standard independent block technique and (b) ours. Observations patterned the same way are assigned to the same block. Observations with no pattern are not assigned to anything (discarded by the method (a)).}
\label{figure:blocks}
\end{figure}
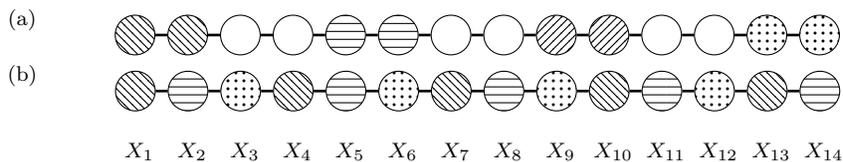

In this section we assume that $X_1, X_2, \dots$ is a realization of a stationary stochastic process,
where \emph{stationarity} means that for all $j \geq 1$ the vector $(X_1, \dots, X_j)$ has the same distribution as $(X_{i+1}, \dots, X_{i+j})$ for all $i \geq 0$.
The following $\alpha$-mixing coefficients are traditionally used to quantify the dependence between the past and the future of the process~\citep{Bradley01}.
\begin{equation}
\alpha_k = \sup_{j \geq 1} \alpha( \,X_1, \dots, X_j \,|\, X_{j+k}, \dots\,).
\end{equation}
A process is called \emph{$\alpha$-mixing} if $\alpha_k$ vanishes in the limit. 

The most popular approach to deal with dependencies in mixing processes is an 
independent block technique~\citep{YuBin01} that approximates the process with a sequence 
of independent blocks.
It is important to note that these independent blocks are understood as 
being independent of each other.

In this section we introduce a different construction based on blocks 
of variables. The blocks we use have orthogonal characteristics to the 
classical ones: variables are independent within the blocks, not across 
them, but dependences can exist between blocks. 
%
The difference is illustrated in the Figure \ref{figure:blocks}.
In (a) we can see the split made by an independent block technique.
The observations are split into contiguous blocks and the blocks are separated from 
each other by a required gap ($2$ in this example). Consequently, to achieve small 
dependence between the blocks, the construction ignores around the half of the observations.
In our construction, (b), we put all observations in one block that are separated 
by a given margin. Thereby, we achieve the small dependence within each block.

Formally, we divide a sample of size $n$ into $\nu$ blocks of length $\mu$, 
such that each block includes variables separated 
by margin $\nu$, \ie the $j$-th block is $\nA_j = (X_j, X_{j+\nu}, \dots, X_{j+(\mu-1)\nu})$.
The first thing to note is that the $\alpha$-separation coefficients of these blocks are related to the $\alpha$-mixing coefficients,
namely $\seqsep{\nA_j} \leq \alpha_\nu$ and therefore $\chi_{\alpha_\nu} \leq \nu$. 
Using Theorem~\ref{theorem:maincovering} with $\gamma = \alpha_\nu$ and $\lambda = t/2$, we get 
the following concentration inequality
\begin{corollary}\label{cor:app_mixing}
Let $X_1, \dots, X_n$ be a realization of a stationary stochastic process with $X_i \in [0,1]$. Then for $t \in [0,1]$
\begin{equation}
\PP{\frac{1}{n}\sum_{i=1}^n X_i - \EE{X_1} > t } \leq \e{ - \frac{\mu t^2}{2} } + \frac{18\sqrt{2}}{\sqrt{t}}n\alpha_{\nu}.
\end{equation}
\end{corollary}

Similar results were proven specifically for stationary processes using the independent block technique. 
For example, Theorem 1.3 from \citep{bosq2012nonparametric} states that for a process 
with $X_i \in [0, 1]$ for each integer $\mu \in [1, \frac{n}{2}]$ and each $t > 0$ with $\nu = \lfloor \frac{n}{2\mu} \rfloor$:
\begin{equation}
\PP{\frac{1}{n}\sum_{i=1}^n X_i - \EE{X_1} > t} \leq 4 \e{-\frac{\mu t^2}{8}} + 22\sqrt{1+\frac{4}{t}}\mu \alpha_{\nu}.
\end{equation}
We can see that both bounds have the same functional form.
Therefore, Theorem \ref{theorem:maincovering} can be seen as a generalization 
of existing concentration results from mixing processes to general sets of random variables.

%
%
\subsection{Lattice models}
Lattice models are used extensively in physics to study different aspects of statistical mechanics, most notably the phenomenon of phase transitions.
For a formal definition, we consider $\sigma$, a configuration of spins indexed by $\bbZ^d$, where each $\sigma_z \in \lbrace \pm 1 \rbrace$ for any site $z \in \bbZ^d$.
Fixing a finite domain $\Delta \subseteq \bbZ^d$, we use a Hamiltonian function $H: \lbrace \pm 1 \rbrace^\Delta \ra \nR_+$ to define the distribution on the set of configurations over $\Delta$ as 
\begin{equation}
\PP{\sigma} = \frac{e^{-\beta H(\sigma)}}{\sum_{\sigma'\in \lbrace \pm 1 \rbrace^\Delta}e^{-\beta H(\sigma')}},
\end{equation}
with a parameter $\beta > 0$ called the inverse temperature.
One example is the Ising model, in which neighboring sites are encouraged to have the same sign,
by the choice $H(\sigma) = \sum_{i,j\in\bbZ^d: \norm{i-j}=1}\sigma_i \sigma_j$ with some fixed boundary condition.

There is a vast literature on the correlations between different sites and their relation to the site distances.
For different Hamiltonians and different temperature regimes the correlations decrease with different rates.
We consider the situations, when the rate is exponential, which is true, for example, for the very general setting with low enough inverse temperature, \eg \citep[Theorem V.2.1]{simon2014statistical}.
Formally, this means that for any sites $z_1, \dots, z_k \in \Delta$ such that $\min \lbrace \norm{z_1- z_2}_1, \dots, \norm{z_1- z_k}_1 \rbrace \geq D$:
\begin{equation}\label{eq:isingcorrelations}
\abs{ \EEbb{\prod_{i=1}^{k} \sigma_{z_i}} - \EEbb{\sigma_{z_1}}\EEbb{\prod_{i=2}^{k} \sigma_{z_i}} } \leq g(k) e^{-\lambda D},
\end{equation}
where $g(k)$ is some polynomial of $k$ and $\lambda > 0$ is a constant that can depend on the inverse temperature.

The following lemma bounds the $\alpha$-separation coefficient between spins in a lattice model using the above bound on their correlations.
\begin{lemma}\label{lemma:isingcorrelations}
Assume a lattice model that fulfills the condition~\eqref{eq:isingcorrelations}.
Then, for any set of sites $z_1, \dots, z_k \in \Delta$ with the property 
that $\norm{z_i- z_j}_1 \geq D$ for all $i\neq j$, 
we have the following bound with $f$ being a polynomial function
\begin{equation}
\seqsep{\lbrace \sigma_{z_1}, \dots, \sigma_{z_k} \rbrace} \leq f(k) e^{-\lambda D}.
\end{equation}
\end{lemma}

\begin{proof}
We  consider a vector $a \in \lbrace \pm 1 \rbrace^k$ and rewrite the following probability:
\begin{align}\label{eq:isingprob}
&\PP{\sigma_{z_1} = a_1, \dots, \sigma_{z_k} = a_k} 
= \frac{\sum_{\sigma:\sigma|_z = a}e^{-\beta H(\sigma)}}{\sum_{\sigma}e^{-\beta H(\sigma)}} = \frac{\sum_{\sigma}\frac{1}{2^k} p(\sigma, a) e^{-\beta H(\sigma)}}{\sum_{\sigma}e^{-\beta H(\sigma)}},
\intertext{%
where $\sigma|_z=(\sigma_{z_1}, \dots, \sigma_{z_k})$ and 
$p(\sigma, a) = \prod_{i=1}^k (1+a_i\sigma_{z_i})$.
In words, $\frac{1}{2^k} p$ is a polynomial of $\sigma$ that equals $1$ when $\sigma|_z = a$ and that vanishes otherwise.
We can rewrite $p(\sigma, a)$ as $1 + \sum_{i} a_i\sigma_{z_i} + \sum_{i<j} a_i a_j \sigma_{z_i}\sigma_{z_j} + \dots + \prod_{i=1}^ka_i\sigma_{z_i}$.
%
Therefore, we can continue \eqref{eq:isingprob} as:}
%
& = \frac{1}{2^k} ( 1 + \sum_{i} a_i\EE{\sigma_{z_i}} + \sum_{i<j} a_i a_j\EE{\sigma_{z_i}\sigma_{z_j}} + \dots + \Big[\prod_{i=1}^k a_i\Big]\EE{\prod_{i=1}^k\sigma_{z_i}} ).
\end{align}
Using the same argument, we can rewrite $\PP{\sigma_{z_1} = a_1}$ and $\PP{\sigma_{z_2} = a_2, \dots, \sigma_{z_k} = a_k}$ to get
\begin{align}
& \PP{\sigma_{z_1} = a_1}\PP{\sigma_{z_2} = a_2, \dots, \sigma_{z_k} = a_k} \\
& = \frac{1}{2^k}( 1 + \sum_{i} a_i \EE{\sigma_{z_i}} + \sum_{1<i<j} a_i a_j \EE
{\sigma_{z_i}\sigma_{z_j}} + \sum_{1<i} a_1 a_i \EE{\sigma_{z_1}}\EE{\sigma_{z_i}}\\
&  + \dots + \Big[\prod_{i=1}^k a_i\Big] \EE{\sigma_{z_1}} \EE{\prod_{i=2}^k\sigma_{z_i}}).
\end{align}
Combining both expressions together, we get
\begin{align}
& \abs{\PP{\sigma_{z_1} = a_1, \dots, \sigma_{z_k} = a_k} - \PP{\sigma_{z_1} = a_1}\PP{\sigma_{z_2} = a_2, \dots, \sigma_{z_k} = a_k}} \\
& \leq \frac{1}{2^k} ( \sum_{1<i} \abs{\EE{\sigma_{z_1}}\EE{\sigma_{z_i}} - \EE{\sigma_{z_1}\sigma_{z_i}}} + \sum_{1<i<j} \abs{\EE{\sigma_{z_1}}\EE{\sigma_{z_i}\sigma_{z_j}} - \EE{\sigma_{z_1}\sigma_{z_i}\sigma_{z_j}}} \\
& + \dots + \abs{\EE{\sigma_{z_1}} \EE{\prod_{t=2}^k\sigma_{z_t}} - \EE{\prod_{t=1}^k\sigma_{z_t}}}) \leq \frac{1}{2^k}h_1(k)e^{-\lambda D}
\end{align}
for some polynomial $h_1$.

Now we can bound the $\alpha$-separation coefficient.
We start with $\alpha(\sigma_{z_1}|\sigma_{z_2}, \dots, \sigma_{z_k})$.
For any sets $A \subseteq \lbrace \pm 1 \rbrace$ and $B \subseteq \lbrace \pm 1 \rbrace^{k-1}$:
\begin{align}
& \PP{\sigma_{z_1}\in A, (\sigma_{z_2}, \dots, \sigma_{z_k}) \in B} - \PP{\sigma_{z_1}\in A }\PP{ (\sigma_{z_2}, \dots, \sigma_{z_k}) \in B} \\
& = \sum_{a_1\in A, b \in B} \PP{\sigma_{z_1} = a, \sigma_{z_2} = b_1 \dots, \sigma_{z_k} = b_{k-1}} - \PP{\sigma_{z_1} = a}\PP{\sigma_{z_2} = b_1, \dots, \sigma_{z_k} = b_{k-1}} \\
& \leq 2^k \frac{1}{2^k}h(k)e^{-\lambda D} = h_1(k)e^{-\lambda D}.
\end{align}
Hence, the supremum over those sets, $\alpha(\sigma_{z_1}|\sigma_{z_2}, \dots, \sigma_{z_k})$, is also bounded by $h_1(k)e^{-\lambda D}$.
The same argument holds for other $\alpha$-dependencies, so we can conclude that
\begin{equation}
\seqsep{\lbrace \sigma_{z_1}, \dots, \sigma_{z_k} \rbrace} \leq \frac{1}{k}\sum_{i=1}^{k-1} \alpha(\sigma_{z_i}|\sigma_{z_{i+1}}, \dots, \sigma_{z_k}) \leq \frac{1}{k}\sum_{i=1}^{k-1} h_i(k-i+1) e^{-\lambda D} = f(k) e^{-\lambda D}.
\end{equation}
\end{proof}

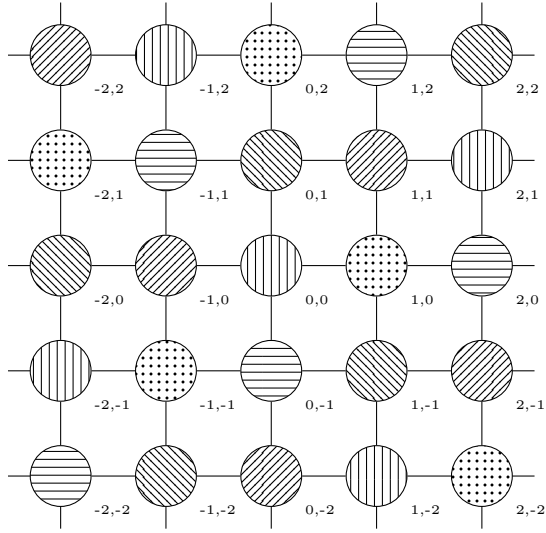
\begin{figure}
\newcommand{\boxnum}{7}
\begin{center}
	\begin{tikzpicture}
	[
		style1/.style={circle,draw,minimum size=23,pattern=north east lines,preaction={fill, white}},
		style2/.style={circle,draw,minimum size=23,pattern=vertical lines,preaction={fill, white}},
		style3/.style={circle,draw,minimum size=23,pattern=dots,preaction={fill, white}},
		style4/.style={circle,draw,minimum size=23,pattern=horizontal lines,preaction={fill, white}},
		style5/.style={circle,draw,minimum size=23,pattern=north west lines,preaction={fill, white}}
	]
		\pgfmathsetmacro{\mult}{1.4}
		\pgfmathsetmacro{\hmult}{\mult/2}
		
		\foreach \x in {0,...,4}
			\draw (\mult*\x, 0)--(\mult*\x, -\hmult) ;
		
		\foreach \x in {0,...,4}
			\draw (\mult*\x, \mult*4)--(\mult*\x, \mult*4+\hmult) ;
					
		\foreach \y in {0,...,4} 
			\draw (0, \mult*\y)--(-\hmult, \mult*\y) ;
			
		\foreach \y in {0,...,4} 
			\draw (\mult*4, \mult*\y)--(\mult*4+\hmult, \mult*\y) ;
			
		\foreach \x in {0,...,4}
			\foreach \y in {0,...,4} {
		   		\pgfmathtruncatemacro{\xx}{\x-2}
		   		\pgfmathtruncatemacro{\yy}{\y-2}
		   		\pgfmathtruncatemacro{\testy}{Mod(\y, 2)}
		   		\pgfmathtruncatemacro{\testx}{Mod(\x, 4)}
		   		\pgfmathtruncatemacro{\start}{Mod(1+2*\y, 5)+1}

		   		\pgfmathtruncatemacro{\st}{Mod(\x-\start, 5)+1}
	   			\node[label={[label distance=1pt]-40:\tiny \xx,\yy},style\st] (\x\y) at (\mult*\x,\mult*\y) {};
			}
		
		\foreach \x in {0,...,4}
			\foreach \y [count=\yi] in {0,...,3}  
				\draw (\x\y)--(\x\yi) (\y\x)--(\yi\x) ;

	\end{tikzpicture}
\end{center}
\caption{A division of a lattice into 5 groups with the distance at least 3 within each group. Each group is represented by a pattern.}
\label{figure:ising}
\end{figure}
As an example of an application of Theorem \ref{theorem:maincovering}, we give a concentration bound for the average magnetization over $\Delta \subseteq \bbZ^2$.
For example, let us take a square set $\Delta$ with side length $L$, meaning that $n = L^2$.
Fix some integer $\nu$ and divide the variables inside the square into $\chi$ groups, $\nG_1, \dots, \nG_\chi$, such that the $\ell_1$ distance between each element inside the group is exactly $\nu$.
The number of such groups, $\chi$, is related to the coloring number of a corresponding distance graph with a set $\lbrace 1, \dots, \nu \rbrace$ over $\bbZ^2$.
For the definitions and overview of the existing bounds on such coloring numbers we refer, for example, to \citep{redl2010distancegraphs}.
An example of a covering is given in the Figure \ref{figure:ising}.
Then, from Lemma \ref{lemma:isingcorrelations}, we get that $\seqsep{G_i} \leq \tilde{g}(\nu)e^{-\lambda \nu}$ for each $i$.
Theorem \ref{theorem:maincovering} then yields a concentration bound on the magnetization inside $\Delta$
\begin{equation}
\PP{\frac{1}{n}\sum_{z\in\Delta} \sigma_z - \EE{\frac{1}{n}\sum_{z\in\Delta} \sigma_z} > t } \leq \e{-2t^2 n /\chi} + f(n)\e{-\lambda \nu},
\end{equation}
for some polynomial $f$.

\subsection{Independent cascade model}
As another example of a setting with small dependencies, we consider an independent cascade model \citep{goldenberg2001talk,goldenberg2001using},
%
as it is commonly used in the study of influence or disease propagation.
%
Consider, for example, a spread of machine failures in a computer network.
Each particular machine can fail at some point in time for reasons that are independent of other machines.
Afterwards, this may cause a failure of the machines that are connected to the first one, for example, if they do a common computation.
Then each of the failed machines can also cause a failure of its neighbors and so on.
The quantities of interest are the final number of failures or the network structures that can minimize the spread of failures.

Mathematically, the model consists of a graph $G$ with vertex set $[n]$, where each vertex $i$ has a binary random variable $Y_i$ associated with it.
The distribution of the variables is defined as a result of the following process.
First, each variable takes a value $1$ ("fires") independently with probability $q$ and $0$ otherwise.
Afterwards, the process proceeds in steps: 
at each step, each variable that fired in the last step has a chance to propagate along each of its outgoing edges to change the value of the neighboring variable.
If the neighboring variable is $1$, then nothing happens.
If it is $0$, then with probability $p$ it changes to $1$.
The process stops when no variable has been switched at the last step.
Then each $Y_i$ is defined to be the final value at node $i$ after the propagation process stops.

In the appendix we prove the following bound on the $\alpha$-separation coefficients for the particular case of a graph $G$, a chain.
\begin{lemma}\label{lemma:cascademixing}
For any index set $\nI \subseteq [n]$ let $d(\nI)$ be the smallest distance in $G$ between any two vertices in $\nI$.
Then for a chain graph $G$, for any set $\nI$ and $0 \leq p < \frac{1}{4}$
\begin{equation}
\seqsep{\nI} \leq |\nI|^2 ( (4p)^{d(\nI)} + 3 p^{d(\nI)}).
\end{equation}
\end{lemma}
Based on the behaviour of this bound we propose the following conjecture for any graph structure.
\begin{conjecture}\label{conjecture:cascade}
For any graph $G$ and any set $\nI$
\begin{equation}
\seqsep{\nI} \leq C (cp)^{d(\nI)},
\end{equation}
where $c$ is a numerical constant and $C$ depends only on the structure of $G$ and is a polynomial in the size of $\nI$.
\end{conjecture}
The dependency graph $\nH$ of $Y_i$'s is a complete graph.
However, assuming the conjecture is true, for the thresholded graph $H_{\gamma}$ for $\gamma = Cp^d$ for some integer $d$ have only edges between variables that have distance less than $d$ in the graph $G$ (similarly to the example in the Section \ref{sec:softcovers}, but now for arbitrary graphs).
For a fixed $d$, let $\chi_d$ be the coloring number of $H_\gamma$ with $\gamma = C(cp)^d$.
Then Theorem \ref{theorem:maincovering} would allow us the following concentration bound on the average number of fired events in the independent cascade model.
For any integer $d$:
\begin{equation}
\PP{\frac{1}{n}\sum_{i=1}^{n}(Y_i - \EE{Y_i}) > t} \leq \e{-2t^2n/\chi_d} + C n (cp)^d.
\end{equation}

\bibliographystyle{imsart-number}
\bibliography{biblio}

\section{Appendix}
In the proof of Theorem \ref{theorem:maincovering} we use the following theorem from \citep{Bradley02}, which allows to construct a tight copy of a single variable.
\begin{theorem}[Theorem 3 from \citep{Bradley02}]\label{theorem:bradley}
Suppose $X \in R^d$ and $Y \in R$ are two random variables.
Suppose $\lambda > 0$ and $1\leq q \leq \infty$, such that $\lambda \leq \norm{Y}_q < \infty$. Then there exists a real-valued random variable $Y^\star$, such that
\begin{enumerate}
\item $Y^\star$ is independent of $X$,
\item $Y^\star$ and $Y$ has the same distribution, and
\item $\PP{\abs{Y - Y^\star} \geq \lambda} \leq 18\left( \frac{\norm{Y}_q}{\lambda}\right)^{q/(2q+1)}\left( \alpha( X | Y ) \right)^{2q/(2q+1)}$.
\end{enumerate}
\end{theorem}
\begin{proof}[Proof of Theorem \ref{theorem:lowerbound}]
We construct the distribution by directly assigning probabilities to each elementary outcome.
Let $\omega \in \left\lbrace 0 , 1 \right\rbrace^n$ and $X = (X_1, \dots, X_n)$.
Then for a fixed $\varepsilon \in [0, 1]$ (to be set later) we define
\begin{equation}
\PP{X = \omega} = \frac{1}{2^n} + s(\omega)\frac{\varepsilon}{2^n},
\end{equation}
where we define $s(w) = 1$ if $\abs{\frac{n}{2} + t - \sum_{i=1}^n \omega_i}$ is even and $s(\omega) = -1$ otherwise.
First, we need to check that it is a valid probability distribution.
For each $\omega$, $\PP{X = \omega} \in [0, 1]$, because $\varepsilon \in [0, 1]$, and 
\begin{align}
\sum_{\omega} \PP{X = \omega} & = \sum_{k=0}^{n} {n \choose k} \left(\frac{1}{2^n} + (-1)^k \frac{\varepsilon}{2^n} \right) \\
& = 1 + \frac{\varepsilon}{2^n} \sum_{k=0}^{n} {n \choose k} (-1)^k = 1,
\end{align}
where the last equality follows from the identity $\sum_{k=0}^{n} {n \choose k} (-1)^k = 0$.

In addition, we need to verify the marginal distributions of $X_i$'s.
Introduce $G_k = \left\lbrace \omega \in \left\lbrace 0 , 1 \right\rbrace^n | \sum_{i=1}^{n} \omega_i = k \text{ or } \sum_{i=1}^{n} \omega_i = n-k \right\rbrace$.
We observe that if we take all $\omega$ within $G_k$, then there is the same number of zeros and ones in the $i$'th coordinate, meaning that the events $\left\lbrace X_i = 0 \right\rbrace$ and $\left\lbrace X_i = 1 \right\rbrace$ get the same probability mass within each $G_k$, because each $\omega$ within $G_k$ has the same probability.
Since the sets $G_k, 0 \leq k \leq \frac{n}{2}$, form a partition of $ \left\lbrace 0 , 1 \right\rbrace^n $, the events $\left\lbrace X_i = 0 \right\rbrace$ and $\left\lbrace X_i = 1 \right\rbrace$ get the same probability mass, that is $\frac{1}{2}$, on the whole space.

Moreover, each set of $X_i$'s of size less than $n$ is independent.
Because of symmetry, we establish this fact for $X_k, \dots, X_n$ for some $k > 1$.
For any $\omega\in\lbrace 0, 1\rbrace^{n-1}$
\begin{align}
\PP{(X_2, \dots, X_n) = \omega} & = \PP{X = (0, \omega)} + \PP{X = (1, \omega)} \\
& = \frac{1}{2^{n-1}} + \frac{\varepsilon}{2^n}(s(0, \omega) + s(1, \omega)) \\
& = \frac{1}{2^{n-1}} = \prod_{i=2}^n \PP{X_i = \omega_i},
\end{align}
because $s(0, \omega) = - s(1, \omega)$ by definition of the function $s$.

Our next step is to compute the sequential $\alpha$-separation coefficient for the defined distribution.
Since our distribution is completely symmetric, the order does not matter and we can consider the natural ordering.
First, we observe that $\alpha(X_k|X_{k+1},\dots,X_n) = 0$ for $k > 1$, because of the independence property proven above.
Recall that
\begin{align}
& \alpha(X_k|X_{k+1},\dots,X_n) \\
& = \sup_{\subalign{ & B_k\subseteq \binary \\ & B\subseteq \binary^{n-k} }} \abs{ \PP{X_k \in B_k, (X_{k+1},\dots,X_n) \in B} - \PP{X_k \in B_k}\PP{(X_{k+1},\dots,X_n) \in B} }.
\end{align}
For fixed $B_k, B$, we can rewrite the probabilities inside the absolute value as
\begin{align}
\sum_{\omega_k\in B_k, \omega\in B}( \PP{(X_k, \dots,X_n) = (\omega_k,\omega)} - \PP{X_k = \omega_k}\PP{(X_{k+1},\dots,X_n) = \omega}) = 0.
\end{align}

We can also compute $\alpha(X_1|X_{2},\dots,X_n)$:
\begin{align}
& \alpha(X_1|X_{2},\dots,X_n) \\ & = \sup_{B_1\subseteq \binary, B\subseteq \binary^{n-1}} \abs{ \sum_{\omega_1\in B_1, \omega\in B}( \PP{X = (\omega_1,\omega)} - \PP{X_1 = \omega_1}\PP{(X_{2},\dots,X_n) = \omega}) } \\
& = \sup_{B_1\subseteq \binary, B\subseteq \binary^{n-1}} \abs{ \sum_{\omega_1\in B_1, \omega\in B}( \frac{1}{2^n} + s(\omega_1, \omega)\frac{\varepsilon}{2^n} - \frac{1}{2^n} )} \\
& = \frac{\varepsilon}{2^n} \sup_{B_1\subseteq \binary, B\subseteq \binary^{n-1}} \abs{ \sum_{\omega_1\in B_1, \omega\in B} s(\omega_1, \omega) }.
\end{align}
Now we need to compute the last supremum.
It can not be achieved for $B_1 = \emptyset$ and $B_1 = \binary$, since the expression under the supremum equals to 0 in those cases.
Then we have only two cases left, $B_1 = \lbrace 0 \rbrace$ and $B_1 = \lbrace 1 \rbrace$, and they achieve the same value as $s(0,\omega) = -s(1,\omega)$.
Therefore, we are left to compute
\begin{equation}
\sup_{B\subseteq \binary^{n-1}} \abs{ \sum_{\omega\in B} s(\omega) }.
\end{equation}
To achieve the supremum, the set $B$ needs to contain all $\omega$ with the property that $\sum_{i=1}^{n-1}\omega_i$ have the same parity (so that we include only the summands with the same sign).
Again, it does not matter which parity we choose, because
\begin{equation}
\sum_{i \text{ odd}} \binom{n-1}{i} = \sum_{i \text{ even}} \binom{n-1}{i} = 2^{n-2}.
\end{equation}
The final expression for the $\alpha$-dependence of our variables is $\seqsep{\nA} = \frac{\varepsilon}{4n}$ and we can set $\varepsilon = 4n\gamma$.

Next, we turn to the lower bound,
\begin{align}
\PP{\bar{X} \geq \frac{n}{2} + t} & = \sum_{k=\frac{n}{2} + t}^{n}\PP{\bar{X} = k} \\
& = \sum_{k=\frac{n}{2} + t}^{n} \sum_{\omega: \sum_{i}\omega_i = k} \PP{X = \omega} \\
& = \frac{1}{2^n} \sum_{k=\frac{n}{2} + t}^{n} {n \choose k} + \frac{\varepsilon}{2^n} \sum_{k=\frac{n}{2} + t}^{n} {n \choose k} (-1)^{\Ind{\abs{\frac{n}{2} + t - k} \text{ is odd}}} \\
& = \frac{1}{2^n} \sum_{k=\frac{n}{2} + t}^{n} {n \choose k} + \frac{\varepsilon}{2^n} {n-1 \choose \frac{n}{2} + t - 1}, \label{eq:decomp-lower}
\end{align}
where in the last line we used the fact that $\sum_{k=0}^{n} {n \choose k} (-1)^k = 0$ and $\sum_{k=0}^{m} {n \choose k} (-1)^k = (-1)^m {n-1 \choose m}$.
Now the first term in (\ref{eq:decomp-lower}) can be lower-bounded as in the proof of Proposition 7.3.2 in \citep{Matousek01}, which gives us the final result:
\begin{equation}
\PP{\bar{X} \geq \frac{n}{2} + t} \geq \frac{1}{15}e^{-16t^2/n} + \frac{4n\seqsep{\nA}}{2^n} {n-1 \choose \frac{n}{2} + t - 1}.
\end{equation}
\end{proof}

\begin{proof}[Proof of Lemma \ref{lemma:cascademixing}]
Let us first fix our notations.
Each vertex $i$ of $G$ is connected to its neighbors $i-1$ and $i+1$, except for $1$ and $n$, which are connected only to $2$ and $n-1$ respectively.
Let $X_i$ be the state of the variables after the initial draw.
Also, for each vertex $i$ define two random variables $L_i$ and $R_i$, which control the propagation, \ie each $L_i$ and $R_i$ take value $1$ with probability $p$ and zero otherwise.
In addition, recursively define axillary variables $A_i = X_i + L_{i+1}A_{i+1}$ with $L_{n+1} = 0$ and $A_{n+1} = 0$.
Similarly, let $B_i = X_i + R_{i-1}B_{i-1}$ with $R_0 = 0$ and $B_0 = 0$.
Then we define the final state variables $Y_i$ as $\min\lbrace 1, A_i + B_i \rbrace$.

Denote the elements of $\nI$ as $i_1, \dots, i_k$, with the indices being sorted in the ascending order.
To prove the lemma it is enough to prove the bound for some order of indices and we will go from left to right, \ie we focus on $\seqsep{Y_{i_1}|\lbrace Y_{i_2}, \dots, Y_{i_k} \rbrace}$.
Moreover, it is enough to bound 
\begin{equation}
\abs{\PP{Y_{i_1} = 0, \dots, Y_{i_k} = 0} - \PP{Y_{i_1} = 0}\PP{Y_{i_2} = 0, \dots, Y_{i_k} = 0}}.
\end{equation}

For this we rewrite the events for any $i$
\begin{align}
\lbrace Y_i = 0 \rbrace &  = \lbrace A_i = 0, B_i = 0 \rbrace \\
& = \lbrace X_i = 0, L_{i+1}A_{i+1} = 0, R_{i-1}B_{i-1} = 0 \rbrace,
\end{align}
%
A useful observation is that we can rewrite some of the events as follows, for any $i, j$,
\begin{align}
\lbrace L_{i+1}A_{i+1} = 0, L_{j+1}A_{j+1} = 0 \rbrace = \lbrace \forall i+1 \leq k \leq j-1, X_k \prod_{t=1}^{k-i}L_{i+t} = 0, L_{j+1}A_{j+1} = 0 \rbrace.
\end{align}
The same can be done for $\lbrace R_{i-1}B_{i-1} = 0, R_{j-1}B_{j-1} = 0 \rbrace$.
Let $C_1 = \lbrace \forall i_1+1 \leq k \leq i_2-1, X_k \prod_{t=1}^{k-i_1}L_{i_1+t} = 0 \rbrace$ and $C_2 = \lbrace \forall i_1+1 \leq k \leq i_2-1, X_k \prod_{t=1}^{i_2-k}R_{i_2-t} = 0 \rbrace$.
Then
\begin{align}
& \PP{Y_{i_1} = 0, \dots, Y_{i_k} = 0} \\
& = \PP{ X_{i_1} = 0, X_{i_2} = 0, C_1, C_2, R_{i_1-1}B_{i_1-1} = 0, L_{i_2+1}A_{i_2+1} = 0, Y_{i_3} = 0, \dots, Y_{i_k} = 0 }. \notag
\end{align}
Let us introduce another version of $B_i$'s as $\tilde{B}_i = X_i + R_{i-1}\tilde{B}_{i-1}$ for $i \geq i_2$ with $\tilde{B}_{i_2} = 0$.
Using those we can define $\tilde{Y}_i = \min\lbrace 1, A_i + \tilde{B}_i \rbrace$ and observe that
\begin{equation}
\lbrace R_{i_2-1}B_{i_2-1} = 0, Y_{i_3} = 0, \dots, Y_{i_k} = 0 \rbrace = \lbrace R_{i_2-1}B_{i_2-1} = 0, \tilde{Y}_{i_3} = 0, \dots, \tilde{Y}_{i_k} = 0 \rbrace,
\end{equation}
which is very convenient, because $\tilde{Y}_i$'s are independent of all random variables with the indices less than $i_2$.

Using all above we can decompose the probabilities as follows.
\begin{align}
& \PP{Y_{i_1} = 0, \dots, Y_{i_k} = 0} = \PP{ X_{i_1} = 0}\PP{ X_{i_2} = 0}\PP{R_{i_1-1}B_{i_1-1} = 0} \PP{C_1, C_2} \\
& \ \ \times \PP{ L_{i_2+1}A_{i_2+1} = 0, \tilde{Y}_{i_3} = 0, \dots, \tilde{Y}_{i_k} = 0 }. \notag \\
& \PP{Y_{i_1} = 0} = \PP{ X_{i_1} = 0}\PP{ L_{i_1+1}A_{i_1+1} = 0}\PP{R_{i_1-1}B_{i_1-1} = 0}. \\
& \PP{Y_{i_2} = 0, \dots, Y_{i_k} = 0} \\
& = \PP{ X_{i_2} = 0}\PP{R_{i_2-1}B_{i_2-1} = 0} \PP{ L_{i_2+1}A_{i_2+1} = 0, \tilde{Y}_{i_3} = 0, \dots, \tilde{Y}_{i_k} = 0 }. \notag
\end{align}
Then we have the following inequality.
\begin{align}
& \abs{\PP{Y_{i_1} = 0, \dots, Y_{i_k} = 0} - \PP{Y_{i_1} = 0}\PP{Y_{i_2} = 0, \dots, Y_{i_k} = 0}} \\
& \leq \abs{ \PP{C_1, C_2} - \PP{L_{i_1+1}A_{i_1+1} = 0}\PP{R_{i_2-1}B_{i_2-1} = 0} }.
\end{align}

Further, introduce $E_L = \prod_{t=1}^{i_2-i_1} L_{i_1+t}$ and $E_R = \prod_{t=1}^{i_2-i_1}R_{i_2-t}$ and observe that all $\PP{E_L = 0, E_R = 1}$, $\PP{E_L = 1, E_R = 0}$ and $\PP{E_L = 1, E_R = 1}$ are bounded by $p^{i_2-i_1}$.
And we can further upper bound
\begin{align}
& \leq \abs{ \PP{C_1, C_2} - \PP{L_{i_1+1}A_{i_1+1} = 0}\PP{R_{i_2-1}B_{i_2-1} = 0} } \\
& \leq \abs{ \PP{C_1, C_2, E_L = 0, E_R = 0} - \PP{L_{i_1+1}A_{i_1+1} = 0, E_L = 0}\PP{R_{i_2-1}B_{i_2-1} = 0, E_R = 0} } \notag\\
& + 3p^{i_2-i_1}, \label{eq:lastline}
\end{align}
where we decomposed all of the probabilities over the events with $E_L$ and $E_R$ and used the above-mentioned upper bounds on the probability for some of them.
Now we introduce the elementary outcomes $\omega_L, \omega_R \in \lbrace 0, 1 \rbrace^{i_2-i_1}$, where $\omega_L$ is an assignment of $L = (L_{i_1+1}, \dots, L_{i_2})$ and $\omega_R$ is an assignment of $R = (R_{i_1}, \dots, R_{i_2-1})$.
Also, denote by $\abs{(\omega_L,\omega_R)}$ a number of ones on the both vectors, \ie $\sum_{k=1}^{i_2-i_1}\omega_L(k) + \sum_{k=1}^{i_2-i_1}\omega_R(k)$.
Then we can continue bounding \eqref{eq:lastline} by
\begin{align}
& | \sum_{\omega_L, \omega_R} ( \PP{C_1, C_2, E_L = 0, E_R = 0, L = \omega_L, R = \omega_R} \\
& - \PP{L_{i_1+1}A_{i_1+1} = 0, E_L = 0, L = \omega_L}\PP{R_{i_2-1}B_{i_2-1} = 0, E_R = 0, R = \omega_R} ) | + 3p^{i_2-i_1} \notag \\
& \leq | \sum_{\stackrel{\omega_L, \omega_R}{\abs{(\omega_L, \omega_R)} < j-i} } ( \PP{C_1, C_2, E_L = 0, E_R = 0, L = \omega_L, R = \omega_R} \label{eq:theverylastdecomposition} \\
& - \PP{L_{i_1+1}A_{i_1+1} = 0, E_L = 0, L = \omega_L}\PP{R_{i_2-1}B_{i_2-1} = 0, E_R = 0, R = \omega_R} ) | \\
& + 2^{2(i_2-i_1)}p^{i_2-i_1} + 3p^{i_2-i_1}, \notag
\end{align}
where the last line follows from the fact that if $\abs{(\omega_L, \omega_R)} < i_2-i_1$, then $\PP{L = \omega_L, R = \omega_R} \leq p^{i_2-i_1}$.
Now let us take a closer look at the event $C_1 = \cap_{k=i+1}^{j-1} D_k$ with $D_k = \lbrace X_k \prod_{t=1}^{k-i_1}L_{i_1+t} = 0 \rbrace$.
When $L = \omega_L$, then either $\prod_{t=1}^{k-i_1}L_{i_1+t} = 0$ (and $D_k$ can be excluded from the intersection) or $\prod_{t=1}^{k-i_1}L_{i_1+t} = 1$, in which case $D_k$ can be replaced with $\lbrace X_k = 0 \rbrace$. Moreover, in the second case, since $\abs{(\omega_L, \omega_R)} < j-i$, we know for sure that $\prod_{t=1}^{i_2-k}R_{i_2-t} = 0$, meaning that the corresponding event can be excluded from $C_2$.
The same considerations apply to $C_2$.
Therefore, for given $(\omega_L, \omega_R)$ with $\abs{(\omega_L, \omega_R)} < i_2-i_1$, each $C_1$ and $C_2$ decompose into intersection of events $\lbrace X_k = 0 \rbrace$ with no index $k$ repeated twice.
Then it follows that both probabilities in \eqref{eq:theverylastdecomposition} decompose into the product of the same multiples and cancel each other out.

From this we can conclude that $\seqsep{Y_{i_1}|\lbrace Y_{i_2}, \dots, Y_{i_k} \rbrace} \leq k((4p)^{i_2-i_1} + 3p^{i_2-i_1}) \leq k((4p)^{d(\nI)} + 3p^{d(\nI)})$ and, consequently, $\alpha(\nI) \leq k^2((4p)^{d(\nI)} + 3p^{d(\nI)})$.

\end{proof}

\end{document}